\documentclass[graybox]{svmult}

\usepackage{amsmath,amssymb}
\usepackage{enumerate,multirow}
\usepackage{mathptmx}       
\usepackage{helvet}         
\usepackage{courier}        
\usepackage{type1cm}        
\usepackage{makeidx}         
\usepackage{graphicx}        
\usepackage{multicol}        
\usepackage[bottom]{footmisc}

\begin{document}
\newcommand{\eps}{\varepsilon}
\newcommand{\Prob}{\mathsf P}
\newcommand{\R}{\mathbb R}
\newcommand{\Beta}{B}

\newcommand*{\abs}[1]{\left\vert#1\right\vert}
\newcommand*{\aabs}[1]{\big\vert#1\big\vert}
\newcommand*{\set}[1]{\left\{#1\right\}}
\newcommand{\ex}[1]{\mathsf{E}\left[\,#1\,\right]}
\newcommand{\norm}[1]{\left\lVert#1\right\rVert}

\title*{Asymptotic Properties of Drift Parameter Estimator Based on
Discrete Observations of Stochastic Differential Equation Driven
by Fractional Brownian Motion}
\titlerunning{Properties of Drift Parameter Estimator for SDE Driven
by fBm}
\authorrunning{Yuliya Mishura et al}
\author{Yuliya Mishura, Kostiantyn Ral'chenko, Oleg Seleznev, and Georgiy Shevchenko}
\institute{Yuliya Mishura \at Department of Probability Theory, Statistics and Actuarial Mathematics, Taras Shevchenko National University of Kyiv, 64 Volodymyrska, 01601 Kyiv, Ukraine \email{myus@univ.kiev.ua}\and Kostiantyn Ral'chenko \at Department of Probability Theory, Statistics and Actuarial Mathematics, Taras Shevchenko National University of Kyiv, 64 Volodymyrska, 01601 Kyiv, Ukraine \email{k.ralchenko@gmail.com} \and Oleg Seleznev \at Institute of  Mathematics and Mathematical Statistics, University of Umea,
SE-90187 Umea, Sweden \email{oleg.seleznjev@math.umu.se} \and Georgiy Shevchenko \at Department of Probability Theory, Statistics and Actuarial Mathematics, Taras Shevchenko National University of Kyiv, 64 Volodymyrska, 01601 Kyiv, Ukraine \email{zhora@univ.kiev.ua}}
\maketitle

\abstract*{In this chapter, we consider a problem of statistical estimation of an unknown drift parameter for a stochastic differential equation driven by fractional Brownian motion. Two estimators based on discrete observations of solution to the stochastic differential equations are constructed. It is proved that the estimators converge almost surely to the parameter value as the observation interval expands and the distance between observations vanishes. A bound for the rate of convergence is given and numerical simulations are presented. As an auxilliary result of independent interest we establish global estimates for fractional derivative of fractional Brownian motion.
}

\abstract{In this chapter, we consider a problem of statistical estimation of an unknown drift parameter for a stochastic differential equation driven by fractional Brownian motion. Two estimators based on discrete observations of solution to the stochastic differential equations are constructed. It is proved that the estimators converge almost surely to the parameter value, as the observation interval expands and the distance between observations vanishes. A bound for the rate of convergence is given and numerical simulations are presented. As an auxilliary result of independent interest we establish global estimates for fractional derivative of fractional Brownian motion.
}
\section*{Introduction}
A fractional Brownian motion (fBm) with the Hurst parameter $H\in(0,1)$ is a centered Gaussian
process $\set{B_t^H,t\ge 0}$  having the covariance $\ex{B_t^HB_s^H}=\frac{1}{2}(s^{2H}+t^{2H}-|t-s|^{2H})$.
Stochastic differential equations driven by an fBm have been a subject of active research for the last two decades. Main reason is that such equations seem to be one of the most suitable tools to model long-range dependence in many applied areas, such as physics, finance, biology,  network studies, etc.

In modeling, the problems of statistical estimation of model parameters are of a particular importance, so the growing number of papers devoted to statistical methods for equations with fractional noise is not surprising. We will cite only few of them, further references can be found in \cite{bishwal,Mishura08,rao-book}. In \cite{kleptsyna}, the authors proposed and studied maximum likelihood estimators for fractional Ornstein--Uhlenbeck process. Related results were obtained in \cite{rao}, where a more general model was considered. In \cite{hu-nualart-stat} the authors proposed a least squares estimator for fractional Ornstein--Uhlenbeck process and proved its asymptotic normality. The estimators constructed in these papers have the same disadvantage: they are based on the whole trajectory of solution to stochastic differential equations, so are not directly implementable. In view of this, estimators based on discrete observations of solutions were proposed in \cite{bertin-tudor-torres, hu-nualart-zhang, tudor-viens, xiao1}. We note that the discretization of the maximum likelihood estimator is extremely involved in the fractional Brownian case, see discussion in \cite{tudor-viens}.

It is worth to mention that papers \cite{hu-nualart-zhang, tudor-viens} deal with the whole range of Hurst parameter $H\in(0,1)$, while other papers cited here investigate only the case $H>1/2$ (which corresponds to long-range dependence);
recall that in the case $H=1/2$, we have a classical diffusion, and there is a huge literature devoted to it; we refer to books \cite{kutoyants,liptser-shiryaev}  for the review of the topic.
We also mention papers \cite{KMM,xiao2}, which deal with parameter estimation in so-called mixed models, involving standard Wiener process along with an fBm.

This chapter deals with statistical estimation of drift parameter for a stochastic differential equation driven by an fBm based on discrete observation of its solution. The model we consider, is fully non-linear, in contrast to \cite{bertin-tudor-torres, hu-nualart-zhang}, which deal with a simple linear model, \cite{xiao1}, devoted to the problem of estimating the parameters for fractional Ornstein–-Uhlenbeck processes from discrete observations, and \cite{tudor-viens}, which studies a model where the fractional Brownian motion enters linearly.
We propose two new estimators and prove their strong consistency under the so-called ``high-frequency data'' assumption that the horizon of observations tends to infinity, while the distance between them goes to zero. Moreover, we obtain almost sure upper bounds for the  rate of convergence of the estimators. The estimators proposed go far away from being maximum likelihood estimators, and this is their crucial advantage, because they keep strong consistency but they are not complicated technically and are convenient for the simulations. This chapter is organized as follows. In Section~\ref{sec:prelim}, we give preliminaries on stochastic integration with respect to an fBm. In this Section, we also give some auxiliary results, which are of independent interest: global estimates for the fractional derivative of an fBm and for  increments of a solution to an fBm-driven stochastic differential equation. In Section~\ref{sec:estim}, we construct  estimators for the drift parameter, prove their strong consistency and establish their rate of convergence. Section~\ref{sec:simul} illustrates our findings with simulation results.

\section{Preliminaries}\label{sec:prelim}

For  a fixed real number $H\in(1/2,1)$, let $\set{B_t^H,t\ge 0}$ be a fractional Brownian motion  with the Hurst parameter $H$ on a complete probability space $(\Omega,\mathcal F,\Prob)$. The integral with respect to the fBm $B^H$ will be understood in the generalized Lebesgue--Stieltjes sense (see \cite{Zah98a}). Its construction uses the fractional derivatives, defined for $a<b$ and $\alpha \in(0,1)$ as
\begin{gather*}
\big(D_{a+}^{\alpha}f\big)(x)=\frac{1}{\Gamma(1-\alpha)}\bigg(\frac{f(x)}{(x-a)^\alpha}+\alpha
\int_{a}^x\frac{f(x)-f(u)}{(x-u)^{1+\alpha}}du\bigg),\\
\big(D_{b-}^{1-\alpha}g\big)(x)=\frac{e^{-i\pi
\alpha}}{\Gamma(\alpha)}\bigg(\frac{g(x)}{(b-x)^{1-\alpha}}+(1-\alpha)
\int_{x}^b\frac{g(x)-g(u)}{(u-x)^{2-\alpha}}du\bigg).
\end{gather*}
Provided that $D_{a+}^{\alpha}f\in L_1[a,b], \ D_{b-}^{1-\alpha}g_{b-}\in
L_\infty[a,b]$, where $g_{b-}(x) = g(x) - g(b)$, the generalized  Lebesgue-Stieltjes integral $\int_a^bf(x)dg(x)$ is defined as
\begin{equation}\label{integr}\int_a^bf(x)dg(x)=e^{i\pi\alpha}\int_a^b\big(D_{a+}^{\alpha}f\big)(x)\big(D_{b-}^{1-\alpha}g_{b-}\big)(x)dx.
\end{equation}
It follows from the H\"older continuity of $B^H$ that for $\alpha\in (1-H,1)$, $D_{b-}^{1-\alpha}B_{b-}^H\in L_\infty[a,b]$ a.s. (we will prove this result in a stronger form further).
Then for a function
$f$ with $D_{a+}^{\alpha}f\in L_1[a,b]$, we can define integral with respect to $B^H$ through  \eqref{integr}:
\begin{equation}\label{integrfbm}
\int_a^b f(x)\,dB^H(x)=e^{i\pi\alpha}\int_a^b (D_{a+}^\alpha f)(x)(D_{b-}^{1-\alpha} B_{b-}^H)(x)\,dx.
\end{equation}

Throughout the paper, the symbol $C$ will denote a generic constant, whose value is not important and may change from one line to another. If a constant depends on some variable parameters, we will put them in subscripts.

\subsection{Estimate of derivative of fractional Brownian motion}
In order to estimate integrals with respect to fractional Brownian motion, we need to estimate the fractional derivative of $B^H$. Let some $\alpha\in(1-H,1/2)$ be fixed in the rest of this paper.
Denote for $t_1<t_2$
$$
Z(t_1,t_2) =  \left(D^{1-\alpha}_{t_2-} B^H_{t_2-}\right)(t_1)= \frac{e^{-i\pi\alpha}}{\Gamma(\alpha)} \left(\frac{B^H_{t_1}-B^H_{t_2}}{(t_2-t_1)^{1-\alpha}}
+(1-\alpha)\int_{t_1}^{t_2}\frac{B^H_{t_1}-B^H_{u}}{(u-t_1)^{2-\alpha}}du\right).
$$
The following proposition is a generalization of \cite[Theorem 3]{KMM}.
\begin{theorem}\label{fracderest}
For any $\gamma>1/2$,  the random variable
\begin{equation}
\xi_{H,\alpha,\gamma}=\sup_{0\le t_1<t_2\le t_1+1}\frac{\abs{Z(t_1,t_2)}}{(t_2-t_1)^{H+\alpha-1} \left(\abs{\log (t_2-t_1)}^{1/2}+1\right)  \left(\log(t_2+2)\right)^\gamma}
\end{equation}
is finite almost surely.
Moreover, there exists $c_{H,\alpha,\gamma}> 0$ such that
$\ex{\exp\set{x \xi_{H,\alpha,\gamma}^2}}<\infty$ for $x<c_{H,\alpha,\gamma}$.
\end{theorem}
\begin{proof}
Let $h(s) = s^{H+\alpha-1}\left(\abs{\log s}^{1/2}+1\right)$, $s>0$. Define for  $T>0$
$$
M_T = \sup_{0\le t_1<t_2\le t_1+1\le T} \frac{\abs{Z(t_1,t_2)}}{h(t_2-t_1)}.
$$
We will first prove that $M_T$ is finite almost surely. Since $\ex{\left(B_t^H-B_s^H\right)^2}= (t-s)^{2H}$, it follows from  \cite[Theorem 4]{marcus} that there exists a random variable $\xi_T$ such that almost surely for all $t_1,t_2$ with $0\le t_1<t_2\le t_1 + 1$
$$
\abs{B_{t_1}^H-B_{t_2}^H}\le \xi_T (t_2-t_1)^H \left(\abs{\log (t_2-t_1)}^{1/2}+1\right).
$$
Then
\begin{gather*}
\abs{Z(t_1,t_2)}\le \frac{\xi_T}{\Gamma(\alpha)} (t_2-t_1)^{H+\alpha-1} \left(\abs{\log (t_2-t_1)}^{1/2}+1\right) + I,
\end{gather*}
where
\begin{gather*}
I = \abs{\int_{t_1}^{t_2}\frac{B^H_{u}-B^H_{t_1}}{(u-t_1)^{2-\alpha}}du}\le \frac{\xi_T}{\Gamma(\alpha)}\int_{t_1}^{t_2}(u-t_1)^{H+\alpha-2}\left(\abs{\log (u-t_1)}^{1/2}+1\right) du\\
\le \frac{\xi_T}{\Gamma(\alpha)}(t_2-t_1)^{H+\alpha-1}\int_0^1 z^{H+\alpha-2} \left(\abs{\log z}^{1/2} + \abs{\log (t_2-t_1)}^{1/2}+1\right)dz\\
\le C\xi_T(t_2-t_1)^{H+\alpha-1} \left( \abs{\log (t_2-t_1)}^{1/2}+1\right),
\end{gather*}
whence finiteness of $M_T$ follows.
Since $M_T$ is a supremum of Gaussian family, Fernique's theorem implies that $\ex{e^{\eps M_T^2}}<\infty$ for some $\eps>0$, in particular, all moments of $M_T$ are finite.

Now observe that from $H$-selfsimilarity of $B^H$ it follows that
for any $a>0$
$$
\set{Z(at_1,at_2),0\le t_1<t_2}\overset{d}{=}\set{a^{H+\alpha-1}Z(t_1,t_2),0\le t_1<t_2}.
$$
Therefore, for any $k\ge 1$
\begin{gather*}
M_{1} \overset{d}{=} \sup_{0\le t_1<t_2\le 1} \frac{2^{-k(H+\alpha-1)}\abs{Z(2^k t_1,2^k t_2)}}{\abs{t_2-t_1}^{H+\alpha-1} \left(\abs{\log (t_2-t_1)}^{1/2}+1\right)}\\
= \sup_{0\le s_1<s_2\le 2^k} \frac{\abs{Z(s_1,s_2)}}{\abs{s_2-s_1}^{H+\alpha-1} \left(\abs{\log (s_2-s_1)-k\log 2}^{1/2}+1\right)}\\
\ge \sup_{0\le s_1<s_2\le s_1+1\le 2^k} \frac{\abs{Z(s_1,s_2)}}{\abs{s_2-s_1}^{H+\alpha-1} \left(\abs{\log (s_2-s_1)-k\log 2}^{1/2}+1\right)}\\
\ge \sup_{0\le s_1<s_2\le s_1+1\le 2^k} \frac{\abs{Z(s_1,s_2)}}{\abs{s_2-s_1}^{H+\alpha-1} \left(\abs{\log (s_2-s_1)}^{1/2}+(k\log 2)^{1/2}+1\right)}\\
\ge \frac{M_{2^k}}{(k\log 2)^{1/2}+1}.
\end{gather*}
Hence, for any $q\ge 1$
$$
\ex{M_{2^k}^q}\le \ex{M_1^q}\left((k\log 2)^{1/2}+1\right)^q.
$$
This implies that for any $p>q/2 +1$
$$
\ex{\sum_{k=1}^{\infty}\frac{M_{2^k}^q}{k^{p}}}= \sum_{k=1}^{\infty}\frac{\ex{M_{2^k}}^q}{k^{p}}\le
C\ex{M_1^q}\sum_{k=1}^{\infty}k^{q/2-p}
<\infty.
$$
In particular, the sum $\sum_{k=1}^{\infty}{\abs{M_{2^k}}^q}{k^{-p}}$ is finite almost surely, so ${M_{2^k}}=o({k^{p/q}})$, $k\to\infty$, a.s. If we choose some $q>(\gamma-1/2)^{-1}$, then $q/2+1<\gamma q$. Hence, we can take some $p\in(q/2+1,\gamma q)$ and arrive at ${M_{2^k}}=o({k^{\gamma}})$, $k\to\infty$, a.s. Thus, the random variable $\zeta = \sup_{k} M_{2^k}k^{-\gamma}$ is finite almost surely.

Obviously, for $t_2\le 2$
$$
\frac{\abs{Z(t_1,t_2)}}{h(t_2-t_1)\log(t_2+2)}\le \frac{M_2}{\log 2}.
$$
Now let $t_2\in(2^{k-1},2^{k}]$ for some $k\ge 2$. Then we have for any $t_1\in [t_2-1,t_2)$
\begin{gather*}
\abs{Z(t_1,t_2)}\le M_{2^k} h(t_2-t_1) \le \zeta k^\gamma h(t_2-t_1)
\le  \zeta \left(\frac{\log t_2}{\log 2}+1\right)^\gamma h(t_2-t_1)\\
 \le 2^\gamma  \zeta (\log t_2)^\gamma  h(t_2-t_1) < 2^\gamma \zeta  (\log (t_2+2))^\gamma  h(t_2-t_1).
\end{gather*}
Consequently, $\xi_{H,\alpha,\gamma} \le \max\set{M_2/\log 2, 2^\gamma \zeta}<\infty$ a.s.

The second statement follows from Fernique's theorem, since $\xi_{H,\alpha,\gamma}$ is a supremum of absolute values of a centered Gaussian family.
\end{proof}

\subsection{Estimates for solution of SDE driven by fractional Brownian motion}

Consider a stochastic differential equation
\begin{equation}\label{2.1}
X_t= X_0 + \int_0^t a(X_s)ds+ \int_0^t b(X_s)dB^H_s,
\end{equation}
where $X_0$ is non-random. In \cite{lyons},
it is shown that this equation has a unique solution under the following assumptions: there exist constants $\delta\in(1/H-1,1]$, $K>0$, $L>0$ and for every $N\ge 1$ there exists $R_N>0$ such that
\begin{enumerate}[(A)]
\item\label{A}
$\abs{a(x)}+\abs{b(x)}\le K$
\quad for all $x\in\R$,
\item\label{B}
$\abs{a(x)-a(y)}+\abs{b(x)-b(y)}\le L\abs{x-y}$
\quad for all $x,y\in\R$,
\item\label{C}
$\abs{b'(x)-b'(y)}\le R_N\abs{x-y}^\delta$
\quad for all $x\in[-N,N], y\in[-N,N]$.
\end{enumerate}

Fix some $\beta \in (1/2,H)$. Denote for $t_1<t_2$ $$\Lambda_\beta(t_1,t_2) = 1\vee \sup_{t_1\le u<v\le t_2} \frac{\abs{Z(u,v)}}{(v-u)^{\beta+\alpha-1}}.$$
Define for $a<b$
$$
\norm{f}_{a,b,\beta} = \sup_{a\le s<t\le b}\frac{\abs{f(t)-f(s)}}{\abs{t-s}^\beta}.
$$

\begin{theorem}\label{Xderivestim}
There exists a constant $M_{\alpha,\beta}$ depending on $\alpha$, $\beta$, $K$, and $L$ such that for any $t_1\ge 0,t_2\in (t_1,t_1+1]$
$$
\abs{X_{t_2} - X_{t_1}}\le M_{\alpha,\beta} \left(\Lambda_\beta(t_1,t_2) (t_2-t_1)^\beta + \Lambda_\beta(t_1,t_2)^{1/\beta}(t_2-t_1)\right).
$$
\end{theorem}
\begin{proof}
The proof follows the lines of \cite[Theorem 2]{hu-nualart}.

Fix $t_1\ge 0$ and $t_2\in(t_1,t_1+1]$. Abbreviate $\Lambda = \Lambda_\beta(t_1,t_2)$.
Take any $s,t$ such that $t_1\le s<t\le t_2$.
Write
$$
\abs{X_t - X_s} \le \int_s^t\abs{a(X_u)}du + \abs{\int_s^t b(X_u)dB^H_u}\le K(t-s) + \abs{\int_s^t b(X_u)dB^H_u}.
$$
Estimate
\begin{equation}\label{intestim}
\begin{gathered}
\abs{\int_s^t b(X_u)dB^H_u}\le \int_s^t \abs{\left(D_{s+}^\alpha b(X)\right)(u)}
\abs{\left(D_{t-}^{1-\alpha} B^H_{t-}\right)(u)}du\\
\le \Lambda \int_s^t
\abs{\left(D_{s+}^{\alpha} b(X)\right)(u)}(t-u)^{\beta+\alpha-1}du.
\end{gathered}
\end{equation}
Now
\begin{gather*}
\abs{\left(D_{s+}^{\alpha} b(X)\right)(u)} \le
 \left(\frac{\abs{b(X_u)}}{(u-s)^\alpha} + \int_s^u \frac{\abs{b(X_u)-b(X_v)}}{(u-v)^{1+\alpha}}dv\right)\\
 \le K(u-s)^{-\alpha} + L \norm{X}_{s,t,\beta}\int_s^u (u-v)^{\beta-\alpha-1}dv\\ \le C_{\alpha,\beta}\left((u-s)^{-\alpha} + \norm{X}_{s,t,\beta}(u-s)^{\beta-\alpha}\right).
\end{gather*}
Hence,
\begin{gather*}
\abs{\int_s^t b(X_s)dB^H_u}\le C_{\alpha,\beta} \Lambda\left(
(t-s)^{\beta} + \norm{X}_{s,t,\beta}(t-s)^{2\beta}\right)
\end{gather*}
and
$$
\norm{X}_{s,t,\beta} \le K_{\alpha,\beta} \Lambda \left(
1 + \norm{X}_{s,t,\beta}(t-s)^{\beta}\right)
$$
with a constant $K_{\alpha,\beta}$ depending only on $\alpha$, $\beta$,  $K$, and $L$.
Setting $\Delta = (2 K_{\alpha,\beta} \Lambda)^{-1/\beta}$, we obtain
$\norm{X}_{s,t,\beta}\le 2K_{\alpha,\beta} \Lambda$ whenever $t-s<\Delta$.

Now, if $0<t_2-t_1\le \Delta$, then $$\abs{X_{t_2}-X_{t_1}}\le \norm{X}_{t_1,t_2,\beta} (t_2-t_1)^\beta\le 2K_{\alpha,\beta}\Lambda (t_2-t_1)^{\beta}.$$
On the other hand, if $t_2-t_1>\Delta$, then, partitioning the interval $[t_1,t_2]$ into $k = [(t_2-t_1)/\Delta]$ parts of length $\Delta$ and, possibly, an extra smaller part, we obtain
\begin{gather*}
\abs{X_{t_2}-X_{t_1}} \le \abs{X_{t_1+\Delta} -X_{t_1}} + 
\dots + \abs{X_{t_1+k\Delta} -X_{t_1+(k-1)\Delta}} + \abs{X_{t_2}-X_{t_1+k\Delta}}\\
\le \left(k+1\right)2K_{\alpha,\beta}\Lambda \Delta^\beta \le 4k K_{\alpha,\beta}\Delta^\beta \le 4K_{\alpha,\beta}\Lambda (t_2-t_1)\Delta^{\beta-1}\\
 = 2(2K_{\alpha,\beta} \Lambda)^{1/\beta} (t_2-t_1).
\end{gather*}
The proof is now complete.
\end{proof}
\begin{corollary}\label{xincrestimate}
For any $\gamma>1/2$, there exist random variables $\xi$ and $\zeta$ such that
for all $t_1\ge 0$, $t_2\in(t_1,t_1+1]$
$$
\abs{X_{t_2}-X_{t_1}}\le \zeta(t_2-t_1)^\beta\left(\log(t_2+2)\right)^\kappa,\quad
\Lambda_\beta (t_1,t_2)\le \xi (\log(t_2+2))^{\kappa\beta},
$$
where $\kappa = \gamma/\beta$.
Moreover, there exists some $c>0$ such that
$\ex{\exp\set{x\xi^{2}}}<\infty$ and $\ex{\exp\set{x\zeta^{2\beta}}}<\infty$ for $x<c$.
In particular, all moments of $\xi$ and $\zeta$ are finite.
\end{corollary}
\begin{proof}
From Theorem~\ref{fracderest} we have for all
$u<v$
\begin{gather*}
Z(u,v)\le  \xi_{H,\alpha,\gamma} (v-u)^{H+\alpha-1}\left(\abs{\log(v-u)}^{1/2}+1\right) \left(\log(v+2)\right)^\gamma\\\le C_{H,\beta}\xi_{H,\alpha,\gamma} (v-u)^{\beta+\alpha-1} \left(\log(v+2)\right)^\gamma,
\end{gather*}
Dividing by $(v-u)^{\beta+\alpha-1}$ and taking supremum over $u,v$ such that $t_1\le u<v\le t_2$, we get
\begin{gather*}
\Lambda_\beta(t_1,t_2)\le  1\vee \left(C_{H,\beta}\xi_{H,\alpha,\gamma} \left(\log(t_2+2)\right)^\gamma\right)\le \left(1\vee C_{H,\beta}\xi_{H,\alpha,\gamma}\right) \left(\log(t_2+2)\right)^\gamma.
\end{gather*}
Further, since $\Lambda_\beta(t_1,t_2)\ge 1$ and $t_2-t_1\le 1$, it follows from Theorem~\ref{Xderivestim} that
$$
\abs{X_{t_2}-X_{t_1}}\le 2M_{\alpha,\beta} \Lambda_\beta(t_1,t_2)^{1/\beta}(t_2-t_1)^{\beta}.
$$
 Hence, the desired statement holds with $\xi = 1\vee C_{H,\beta}\xi_{H,\alpha,\gamma}$ and $\zeta = 2M_{\alpha,\beta}\xi^{1/\beta}$.
\end{proof}

The following lemma gives a particular case of Corollary~\ref{xincrestimate}, suitable for our needs.
Let $\gamma>1/2$ and $\kappa=\gamma/\beta$ be fixed, $\xi$ and $\zeta$ will be the corresponding random variables from Corollary~\ref{xincrestimate}.
\begin{lemma}
For any $n\ge 2$ and any $t_1,t_2\in[0,2^{n}]$ such that $t_1<t_2\le t_1+1$
$$
\abs{X_{t_2}-X_{t_1}}\le \zeta n^\kappa (t_2-t_1)^\beta,
\quad \Lambda_\beta(t_1,t_2) \le \xi n^\gamma.
$$
\end{lemma}
\begin{proof}
In this case
$$
\log(t_2+2)\le \log(2^{n}+2)\le\log 2^{n+1}= (n+1)\log 2\le n,
$$
whence the statement follows.
\end{proof}

\section{Drift parameter estimation}\label{sec:estim}
Now we turn to problem of drift parameter estimation in equations of
type \eqref{2.1}.
Let $(\Omega,\mathcal F)$ be a measurable space and $X:\Omega \to C[0,\infty)$ be a stochastic process. Consider a family of probability measures $\set{\Prob^\theta,\theta\in\R}$ on $(\Omega,\mathcal F)$ such that for each $\theta\in\R$, $\mathcal F$ is $\Prob^\theta$-complete, and there is an fBm $\set{B^{H,\theta}_t,t\ge 0}$
on $(\Omega,\mathcal F,\Prob^\theta)$ such that $X$ solves a parametrized version of \eqref{2.1}
\begin{equation}\label{theta-sde}
X_t= X_0 + \theta \int_0^t a(X_s)ds+ \int_0^t b(X_s)dB^{H,\theta}_s.
\end{equation}
Our main problem is to construct an estimator for $\theta$ based on discrete observations of $X$.
Specifically, we will assume that for some $n\ge 1$ we observe the values
$X_{t_k^n}$ at the
following uniform  partition of $[0,2^n]$: $t_k^n=k 2^{-n}$, $k=0,1,\dots,2^{2n}$.

To simplify the notation, in the following we will fix an arbitrary $\theta\in \R$ and denote simply $B^{H,\theta}=B^H$, $\Prob^\theta = \Prob$. We also fix the parameters $\alpha\in (1-H,1/2)$, $\beta\in (1-\alpha,H)$, $\gamma >1/2$, and $\kappa = \gamma/\beta$. Finally, with a slight abuse of notation, let  $\xi$ and $\zeta$ be the random variables from Corollary~\ref{xincrestimate} applied to equation \eqref{theta-sde}.

In order to construct a consistent estimator, we need a lemma concerning the discrete approximation of
integrals in \eqref{theta-sde}.

\begin{lemma}\label{2.cor2}
For all $n\ge 1$ and $k=1,2,\dots,2^{2n}$
$$
\abs{\int_{t^n_{k-1}}^{t^n_k}
\left(a(X_u)-a(X_{t_{k-1}^n})\right)du}
\le C\zeta n^{\kappa}2^{-n(\beta+1)}
$$
and
$$
\abs{\int_{t^n_{k-1}}^{t^n_k}
\left(b(X_u)-b(X_{t_{k-1}^n})\right)dB^{H,\theta}_u}
\le C\xi\zeta n^{\gamma+\kappa} 2^{-2n\beta}.
$$
\end{lemma}
\begin{proof}
Write
\begin{gather*}
\abs{\int_{t^n_{k-1}}^{t^n_k}
\left(a(X_u)-a(X_{t_{k-1}^n})\right)du}
\le \int_{t^n_{k-1}}^{t^n_k}
\abs{a(X_u)-a(X_{t_{k-1}^n})}du\\
\le
K\zeta n^{\kappa} \int_{t^n_{k-1}}^{t^n_k} (u-t_{k-1}^n)^\beta du
\le C \zeta n^{\kappa} (t_k^n-t_{k-1}^n)^{\beta+1} = C\zeta n^{\kappa}2^{-n(\beta+1)}.
\end{gather*}
Similarly to \eqref{intestim},
\begin{gather*}
\abs{\int_{t^n_{k-1}}^{t^n_k}
\left(b(X_u)-b(X_{t_{k-1}^n})\right)dB^H_u}
\\\le \Lambda_\beta(t_{k-1}^n,t_k^n) \int_{t^n_{k-1}}^{t^n_k}
\abs{D_{t^n_{k-1}+}^\alpha (b(X)-b(X_{t_{k-1}^n}))(u)}(t^n_{k}-u)^{\beta+\alpha-1}du\\
\le \xi n^\gamma \int_{t^n_{k-1}}^{t^n_k}
\abs{D_{t^n_{k-1}+}^\alpha (b(X)-b(X_{t_{k-1}^n}))(u)}(t^n_{k}-u)^{\beta+\alpha-1}du,
\end{gather*}
and
\begin{gather*}
\abs{D_{t^n_{k-1}+}^\alpha (b(X)-b(X_{t_{k-1}^n}))(u)} \le
\frac{\aabs{b(X_{u})-b(X_{t_{k-1}^n})}}{(u-t_{k-1}^n)^\alpha} + \int_{t^n_{k-1}}^{u} \frac{\abs{b(X_{u})-b(X_{v})}}{(u-v)^{1+\alpha}}dv\\
\le K \zeta n^\kappa \left(u-t^n_{k-1}\right)^{\beta-\alpha} +
K\zeta n^\kappa \int_{t^n_{k-1}}^{u} (u-v)^{\beta-\alpha-1}dv \le C\zeta n^\kappa\left(u-t^n_{k-1}\right)^{\beta-\alpha}.
\end{gather*}
Then we can write the estimate
\begin{gather*}
\abs{\int_{t^n_{k-1}}^{t^n_k}
\left(b(X_u)-b(X_{t_{k-1}^n})\right)dB^H_u}
\le C\xi\zeta n^{\gamma+\kappa} \int_{t^n_{k-1}}^{t^n_k}
(u-t^n_{k-1})^{2\beta-1}du\\
\le C\xi\zeta n^{\gamma+\kappa} \left(t^n_k-t^n_{k-1}\right)^{2\beta}=C\xi\zeta n^{\gamma+\kappa} 2^{-2n\beta},
\end{gather*}
which finishes the proof.
\end{proof}

Now we are ready to construct consistent estimators for $\theta$. In order to proceed, we need a technical assumption, in addition to conditions~\eqref{A}--\eqref{C}:
\begin{enumerate}[(A)]
\setcounter{enumi}{3}
\item there exist a constant $M>0$ such that for all $x\in\R$
$$
\abs{a(x)}\ge M,
\qquad
\abs{b(x)}\ge M.
$$
\end{enumerate}

Consider now the following estimator:
$$
\hat\theta_n^{(1)}=\frac{\sum_{k=1}^{2^{2n}-1}\left(t_k^n\right)^{\lambda}
\left(2^n-t_k^n\right)^{\lambda}
b^{-1}\left(X_{t_{k-1}^n}\right)
\left(X_{t_{k}^n}-X_{t_{k-1}^n}\right)}{\sum_{k=1}^{2^{2n}-1}\left(t_{k}^n\right)^{\lambda}
\left(2^n-t_{k}^n\right)^{\lambda}
b^{-1}\left(X_{t_{k-1}^n}\right)
a\left(X_{t_{k-1}^n}\right)\frac{1}{2^n}},
$$
where $\lambda = 1/2-H$.
\begin{theorem}\label{theta1cons}
With probability one,
$\hat\theta_n^{(1)}\to\theta$, $n\to\infty$. Moreover, there exists a random variable $\eta$ with all finite moments such that
$\abs{\hat\theta_n^{(1)}-\theta}\le \eta n^{\kappa+\gamma} 2^{-\rho n}$, where $\rho = (1-H)\wedge (2\beta-1)$.
\end{theorem}
\begin{proof}
It follows from~\eqref{theta-sde} that
\begin{equation*} 
\begin{split}
X_{t_{k}^n}-X_{t_{k-1}^n}
&=\theta\int_{t_{k-1}^n}^{t_{k}^n}a(X_v)dv+\int_{t_{k-1}^n}^{t_{k}^n}b(X_v)dB^H_v\\
&=\theta\int_{t_{k-1}^n}^{t_{k}^n}a\left(X_{t_{k-1}^n}\right)dv
+\theta\int_{t_{k-1}^n}^{t_{k}^n}\left(a(X_v)-a\left(X_{t_{k-1}^n}\right)\right)dv\\
&\quad+\int_{t_{k-1}^n}^{t_{k}^n}b\left(X_{t_{k-1}^n}\right)dB^H_v
+\int_{t_{k-1}^n}^{t_{k}^n}\left(b(X_v)-b\left(X_{t_{k-1}^n}\right)\right)dB^H_v.
\end{split}
\end{equation*}
Then
\begin{equation*}
\hat\theta_n^{(1)}=\theta+\frac{B_n+E_n+D_n}{A_n},
\end{equation*}
where
\begin{align*}
A_n&=2^{n(2H-3)}\sum_{k=1}^{2^{2n}-1}\left(t_{k}^n\right)^{\lambda}
\left(2^n-t_{k}^n\right)^{\lambda}a\left(X_{t_{k-1}^n}\right)
b^{-1}\left(X_{t_{k-1}^n}\right),\\
B_n&= 2^{2n(H-1)} \theta \sum_{k=1}^{2^{2n}-1}\left(t_{k}^n\right)^{\lambda}
\left(2^n-t_{k}^n\right)^{\lambda}b^{-1}\left(X_{t_{k-1}^n}\right)
\int_{t_{k-1}^n}^{t_{k}^n}\left(a(X_v)-a\left(X_{t_{k-1}^n}\right)\right)dv,\\
E_n&={2^{2n(H-1)}}\sum_{k=1}^{2^{2n}-1}\left(t_{k}^n\right)^{\lambda}\left(2^n-t_{k}^n\right)^{\lambda}
\left(B^H_{t_{k}^n}-B^H_{t_{k-1}^n}\right),\\
D_n&={2^{2n(H-1)}}\sum_{k=1}^{2^{2n}-1}\left(t_{k}^n\right)^{\lambda}
\left(2^n-t_{k}^n\right)^{\lambda}b^{-1}\left(X_{t_{k-1}^n}\right)
\int_{t_{k-1}^n}^{t_{k}^n}\left(b(X_v)-b\left(X_{t_{k-1}^n}\right)\right)dB^H_v.
\end{align*}


It is not hard to show that the sequence
\begin{gather*}
\gamma_n=2^{n(2H-3)}\sum_{k=1}^{2^{2n}-1}\left(t_k^n\right)^{\lambda}
\left(2^n-t_k^n\right)^{\lambda} = \sum_{k=1}^{2^{2n}-1}\left(\frac{k}{2^{2n}}\right)^{\lambda}
\left(1-\frac{k}{2^{2n}}\right)^{\lambda}\frac{1}{2^{2n}}
\end{gather*}
converges to $\int_0^1x^{\lambda}(1-x)^{\lambda}dx=\Beta(1+\lambda,1+\lambda)$, hence, is bounded and uniformly positive.

Indeed,
$h(x)=x^{\lambda}(1-x)^{\lambda}$
increases for $x\in\left(0,\frac12\right]$,
then
$$
\int_0^{\frac12}h(x)dx
=\sum_{k=0}^{2^{2n-1}-1}\int_{\frac{k}{2^{2n}}}^{\frac{k+1}{2^{2n}}}h(x)dx
<\int_0^{\frac1{2^{2n}}}h(x)dx
    +\sum_{k=1}^{2^{2n-1}}h\left(\frac{k}{2^{2n}}\right)\frac1{2^{2n}}.
$$
On the other hand,
$$
\int_0^{\frac12}h(x)dx
=\sum_{k=1}^{2^{2n-1}}\int_{\frac{k-1}{2^{2n}}}^{\frac{k}{2^{2n}}}h(x)dx
>\sum_{k=1}^{2^{2n-1}}h\left(\frac{k}{2^{2n}}\right)\frac1{2^{2n}}.
$$
So
$$
0<\int_0^{\frac12}h(x)dx-\sum_{k=1}^{2^{2n-1}}h\left(\frac{k}{2^{2n}}\right)\frac1{2^{2n}}
<\int_0^{\frac1{2^{2n}}}h(x)dx
\rightarrow0,\ n\to\infty.
$$
Hence,
$$\sum\limits_{k=1}^{2^{2n-1}}h\left(\frac{k}{2^{2n}}\right)\frac1{2^{2n}}\to\int_0^{\frac12}h(x)dx,
n\to\infty.$$
Similarly one can prove that
$$\sum\limits_{k=2^{2n-1}+1}^{2^{2n}-1}h\left(\frac{k}{2^{2n}}\right)\frac1{2^{2n}}\to\int_{\frac12}^1h(x)dx,\
n\to\infty.$$

By assumption (D),
$a(x)b^{-1}(x)$
is bounded away from zero and keeps its sign.
Therefore,
$$
\liminf_{n\to\infty}\abs{A_n}
\ge MK^{-1}\lim_{n\to\infty}\gamma_n
=MK^{-1}\Beta(1+\lambda,1+\lambda)>0.
$$
So it is sufficient to estimate
$B_n$, $E_n$, and $D_n$.

By Lemma~\ref{2.cor2},
\begin{align*}
\abs{B_n}&\le  C \abs{\theta}\zeta n^{\kappa} M^{-1} 2^{n(2H-\beta-3)} \sum_{k=1}^{2^{2n}-1}\left(t_{k}^n\right)^{\lambda}
\left(2^n-t_{k}^n\right)^{\lambda}
\le C_{\theta} \zeta n^{\kappa} 2^{-n\beta};\\
\abs{D_n}&\le C\xi\zeta n^{\gamma+\kappa} {M^{-1}}{2^{n(2H-2-2\beta)}}\sum_{k=1}^{2^{2n}-1}\left(t_{k}^n\right)^{\lambda}
\left(2^n-t_{k}^n\right)^{\lambda}
\le C \xi\zeta n^{\gamma+\kappa} 2^{n(1-2\beta)}.
\end{align*}
Finally we estimate $E_n$. Start by writing
$$
\ex{E_n^2}
=2^{4n(H-1)}\ex{\left(\sum_{k=1}^{2^{2n}-1}\int_{t_{k-1}^n}^{t_{k}^n}
\left(t_{k}^n\right)^{\lambda}\left(2^n-t_{k}^n\right)^{\lambda}
dB^H_s\right)^2}.
$$
According to~\cite[Corollary 1.9.4]{Mishura08}, for $f\in L_{1/H}[0,t]$ there exist a constant $C_H>0$ such that
$$
\ex{\left(\int_0^tf(s)dB^H_s\right)^2}\le C_H\left(\int_0^t\abs{f(s)}^{1/H}ds\right)^{2H}.
$$
Hence,
\begin{gather*}
\ex{E_n^2}
\le C2^{4n(H-1)}\left(\sum_{k=1}^{2^{2n}-1}\int_{t_{k-1}^n}^{t_{k}^n}
\left(t_{k}^n\right)^{\lambda/H}\left(2^n-t_{k}^n\right)^{\lambda/H}
ds\right)^{2H}\\=C 2^{2n(H-1)}\left(\sum_{k=1}^{2^{2n}-1}
\left(\frac{k}{2^{2n}}\right)^{\lambda/H}\left(1-\frac{k}{2^{2n}}\right)^{\lambda/H}
\frac{1}{2^{2n}}\right)^{2H}.
\end{gather*}
As above,
$$
\sum_{k=1}^{2^{2n}-1}
\left(\frac{k}{2^{2n}}\right)^{\lambda/H}\left(1-\frac{k}{2^{2n}}\right)^{\lambda/H}
\frac{1}{2^{2n}}
\rightarrow
\Beta\left(1+\lambda/H,1+\lambda/H\right),\ n\to\infty,
$$
which implies that
$\ex {E_n^2}\le C 2^{2n(H-1)}$. Since $E_n$ is Gaussian,
$\ex {\abs{E_n}^p}\le C_p 2^{pn(H-1)}$ for any $p\ge 1$.
Therefore, for any $\nu>1$
$$
\ex{\sum_{n=1}^\infty \frac{\abs{E_n}^p}{n^\nu 2^{pn(H-1)}}} =
\sum_{n=1}^\infty \frac{\ex{\abs{E_n}^p}}{n^\nu 2^{pn(H-1)}}\le
C_p \sum_{n=1}^{\infty} n^{-\nu}<\infty.
$$
Consequently, $$
\xi': = \sup_{n\ge 1}\frac{\abs{E_n}}{n^{\nu/p} 2^{n(H-1)}}<\infty$$
almost surely, moreover, by Fernique's theorem, all moments of $\xi'$ are finite.

Let us summarize the estimates:
\begin{gather*}
\abs{B_n}\le C_{\theta} \zeta n^{\kappa} 2^{-n\beta},\
\abs{D_n} \le C \xi\zeta n^{\gamma+\kappa} 2^{n(1-2\beta)},\
\abs{E_n}\le \xi' n^{\delta} 2^{n(H-1)},
\end{gather*}
where $\delta>0$ can be taken arbitrarily small. We have $-\beta<-1/2<H-1$, $-\beta<1-2\beta$, so $\abs{B_n}$ is of the smallest order. Which of the remaining two estimates wins, depends on values of $\beta$ and $H$: for $H$ close to $1/2$, $1-2\beta$ is close to $0$, while $H-1$ is close to $-1/2$; for $\beta$ close to $1$, $1-2\beta$ is close to $-1$, while $H-1$ is close to $0$.
Thus, we arrive to
\begin{gather*}
\abs{B_n}+\abs{E_n}+\abs{D_n}\le \eta n^{\gamma+\kappa}  2^{-\rho n},
\end{gather*}
where $\eta\le C_\theta (\zeta+\xi\zeta+\xi')$, so all its moments are finite. The proof is now complete.
\end{proof}

Consider a simpler estimator:
$$
\hat\theta_n^{(2)}=\frac{\sum_{k=1}^{2^{2n}-1}
b^{-1}\left(X_{t_{k-1}^n}\right)
\left(X_{t_{k}^n}-X_{t_{k-1}^n}\right)}%
{\frac{1}{2^n}\sum_{k=1}^{2^{2n}-1}
b^{-1}\left(X_{t_{k-1}^n}\right)
a\left(X_{t_{k-1}^n}\right)}.
$$
This is a discretized maximum likelihood estimator for $\theta$ in equation \eqref{2.1}, where $B^H$ is replaced by Wiener process. Nevertheless, this estimator is consistent as well. Namely, we have the following result, whose proof is similar to that of Theorem~\ref{theta1cons}, but is much simpler, so we omit it.
\begin{theorem}\label{theta2cons}
With probability one,
$\hat\theta_n^{(2)}\to\theta$, $n\to\infty$.  Moreover, there exists a random variable $\eta'$ with all finite moments such that
$\abs{\theta_n^{(2)}-\theta}\le \eta' n^{\kappa+\gamma}2^{-\rho n}$.
\end{theorem}
\begin{remark}
Using Theorem~\ref{fracderest}, it can be shown with some extra technical work that
\begin{equation}\label{thetarate}
\abs{\theta_n^{(i)}-\theta}\le \eta_1 n^{\mu}2^{- \tau n},\ i=1,2,
\end{equation}
where $\mu = 1/2 + \gamma(1+1/H)$, $\tau = (2H-1)\wedge (1-H)$; $\eta_1$ is a random variable, for which there exists some $c_\theta>0$  such that  $\ex{\exp\set{x\eta_1^{1+1/H}}}<\infty$ for $x<c_\theta$. Moreover, in order to estimate the estimators reliability, the constant $c_\theta$ can be computed explicitly in terms of $H,K,L,\theta$. However, we will not undertake this tedious task.
\end{remark}
\section{Simulations}\label{sec:simul}
In this section we illustrate quality of the estimators with the help of simulation experiments. For each set of parameters, we simulate 20 trajectories of the solution to \eqref{theta-sde}. Then for each of estimators $\theta_n^{(i)}$, $i=1,2$, we compute the average relative error ${\delta}_n^{(i)}$, i.e. the average of values $\abs{\theta_n^{(i)}-\theta}/\theta$.
We remind that for a particular value of $n$ we take $2^{2n}$ equidistant observations of the process on the interval $[0,2^n]$.

We start with a case of relatively ``tame'' coefficients $a(x) = 2\sin x + 3$, $b(x) = 2\cos x + 3$. We choose $\theta = 2$.

\renewcommand{\arraystretch}{1.5}
\setlength{\tabcolsep}{3pt}
\begin{table}
\begin{center}
\begin{tabular}{|c|c|c|c|c|c|c|c|c|}\hline
\multirow{2}{*}{$n$} & \multicolumn{2}{|c|}{$H=0.6$} & \multicolumn{2}{|c|}{$H=0.7$} & \multicolumn{2}{|c|}{$H=0.8$} & \multicolumn{2}{|c|}{$H=0.9$}\\\cline{2-9}
 & ${\delta}_n^{(1)}$ & ${\delta}_n^{(2)}$ & ${\delta}_n^{(1)}$ & ${\delta}_n^{(2)}$ & ${\delta}_n^{(1)}$ & ${\delta}_n^{(2)}$ & ${\delta}_n^{(1)}$ & ${\delta}_n^{(2)}$ \\\hline
3 & 0.093 & 0.093 & 0.097 & 0.094 & 0.098 & 0.096 & 	0.091 & 0.092\\
\hline 4 & 0.043 & 0.044 & 0.047 & 0.047 & 0.046 & 0.046 & 0.048 & 0.047\\\hline
5 & 0.025 & 0.024 & 0.027 & 0.027 & 0.029 & 0.029 & 0.028 & 0.028\\\hline
6 & 0.011 & 0.011 & 0.012 & 0.012 & 0.016 & 0.016 & 0.016 & 0.016\\
\hline
\end{tabular}
\caption{Relative errors of estimators $\theta_n^{(i)}$, $i=1,2$, for $a(x) = 2\sin x + 3$, $b(x) = 2\cos x + 3$, $\theta = 2$.}
\end{center}
\end{table}

The first observation is that the estimators have similar performance. This means that  $\theta_n^{(2)}$ is preferable to $\theta_n^{(1)}$, since it does not involve $H$ (which might be unknown) and is computable faster (for $n=6$, computation of $\theta_n^{(1)}$ takes $473$~microseconds on Intel Core i5-3210M processor, while that of $\theta_n^{(2)}$ takes $32$~microseconds).

The second observation is that the estimate \eqref{thetarate} of the convergence rate is probably not optimal; it seems that the convergence rate of convergence is around $2^{-n}$, in particular, it is independent of $H$.

Now take worse coefficients $a(x) = 2\sin x + 2.1$, $b(x) = 2\cos x + 2.1$; again $\theta = 2$.

\begin{table}
\begin{center}
\begin{tabular}{|c|c|c|c|c|c|c|c|c|}\hline
\multirow{2}{*}{$n$} & \multicolumn{2}{|c|}{$H=0.6$} & \multicolumn{2}{|c|}{$H=0.7$} & \multicolumn{2}{|c|}{$H=0.8$} & \multicolumn{2}{|c|}{$H=0.9$}\\\cline{2-9}
 & ${\delta}_n^{(1)}$ & ${\delta}_n^{(2)}$ & ${\delta}_n^{(1)}$ & ${\delta}_n^{(2)}$ & ${\delta}_n^{(1)}$ & ${\delta}_n^{(2)}$ & ${\delta}_n^{(1)}$ & ${\delta}_n^{(2)}$ \\\hline
3 & 0.17 & 0.18 & 0.18 & 0.19 & 0.18 & 0.18 & 0.17 & 0.17\\
\hline 4 & 0.096 & 0.097 & 0.099 & 0.102 & 0.099 & 0.106 & 0.095 & 0.099\\
\hline
5 & 0.045 & 0.045 & 0.052 & 0.052 & 0.051 & 0.053 & 0.046 & 0.046\\\hline
6 & 0.024 & 0.024 & 0.021 & 0.021 & 0.027 & 0.028 & 0.033 & 0.033\\
\hline
\end{tabular}
\caption{Relative errors of $\theta_n^{(i)}$, $i=1,2$, for $a(x) = 2\sin x + 2.1$, $b(x) = 2\cos x + 2.1$, $\theta = 2$.}
\end{center}
\end{table}

The relative errors have increased two to three times due to the coefficients approaching zero closer. Also observe that in this case the convergence rate seems better than the estimate \eqref{thetarate}.

Further we show that, despite condition (D) might seem too restrictive, certain condition that the coefficients are non-zero is required.

To illustrate this, take first $a(x) = 2\cos x+1$, $b(x) = 2\sin x + 3$, $\theta=2$. From the first sight, it seems that the estimators should work fine here. Such intuition is based on the observation that the proof of Theorem~\ref{theta1cons} relies on sufficiently fast convergence of the denominator to $+\infty$, which somehow should follow from the fact that positive values of the ratio $a(x)/b(x)$ are overwhelming. 
Unfortunately, this intuition is wrong. Here are ten values of the estimator $\theta_n^{(1)}$ for $H=0.7$, $n=6$:   $1.3152$, $0.6402$, $1.9676$, $0.9600$, $0.4627$, $4.7017$, $0.8386$, $0.8425$, $1.0247$, $0.3902$. Values of the estimator $\theta_n^{(2)}$ are also useless:  $0.7499$, $0.4081$, $1.0179$, $0.5725$, $0.2668$, $-3.1605$, $0.6556$, $0.4413$, $0.5586$, $0.2115$.

Now ``swap'' the roles of $a$ and $b$ by taking $a(x) = 2\cos x+3$, $b(x) = 2\sin x +1$  and keeping other parameters, i.e. $\theta=2$, $H=0.7$, $n=6$. As before, we compute ten values of the estimator $\theta_n^{(1)}$:   $1.5010$, $1.9824$, $2.0666$, $2.0087$, $1.6751$, $1.8802$, $2.1087$, $2.3519$, $2.0160$, $2.0442$; and ten values of  $\theta_n^{(2)}$:  $2.2076$, $1.9853$, $2.0975$, $2.0109$, $1.1202$, $1.8768$, $2.0964$, $2.6175$, $2.0176$, $2.045$. Although the performance of the estimators is mediocre, it has clearly  improved significantly compared to the previous case. We can conclude that sign changes of the coefficient $a$ to zero affect the performance much stronger than those of the coefficient $b$.

\end{document}